\documentclass{amsart}
\usepackage{amsmath,amssymb}
\usepackage{graphicx,subfigure}
\newtheorem{theorem}{Theorem}[section]

\newtheorem{lemma}[theorem]{Lemma}

\theoremstyle{definition}

\theoremstyle{remark}
\newtheorem{remark}[theorem]{Remark}

\numberwithin{equation}{section}

\newcommand{\al}{\alpha}
\newcommand{\be}{\beta}
\newcommand{\de}{\delta}
\newcommand{\ep}{\epsilon}

\newcommand{\ga}{\gamma}
\newcommand{\ka}{\kappa}
\newcommand{\la}{\lambda}

\newcommand{\si}{\sigma}
\newcommand{\te}{\theta}

\newcommand{\De}{\Delta}
\newcommand{\Ga}{\Gamma}
\newcommand{\La}{\Lambda}

\newcommand{\tJ}{\widetilde{J}}
\newcommand{\tT}{\widetilde{T}}
\newcommand{\tsi}{\widetilde{\si}}

\newcommand{\tga}{\widetilde{\ga}}

\def\RR{\mathbb{R}}

\def\ZZ{\mathbb{Z}}

\def\TT{\mathbb{T}}

\newcommand{\cB}{{\mathcal B}}
\newcommand{\cC}{{\mathcal C}}

\newcommand{\cE}{{\mathcal E}}

\newcommand{\cO}{{\mathcal O}}

\newcommand{\cS}{{\mathcal S}}
\newcommand{\cT}{{\mathcal T}}

\newcommand{\pd}{\partial}
\newcommand\minus\backslash
\newcommand{\id}{{\rm id}}

\newcommand{\ms}{\mspace{1mu}}
\newcommand\lan\langle
\newcommand\ran\rangle

\DeclareMathOperator\Div{div}

 \DeclareMathOperator\curl{curl}

\DeclareMathOperator\dist{dist}

\def\sur{{\cS_{\tga,\ep}}}
\def\es{\RR/\ell\ZZ}

\newcommand\DD{\mathbb D}
\renewcommand\leq\leqslant
\renewcommand\geq\geqslant
\newlength{\intwidth}

\DeclareMathOperator\T{\mathbf{t}}
\DeclareMathOperator\N{\mathbf{n}}
\DeclareMathOperator\B{\mathbf{b}}

\def\surf{{\cS_{\ga,\ep}}}

\begin{document}

\title[A problem of Ulam about
  magnetic fields generated by wires]{A problem of Ulam about
  magnetic fields generated by knotted wires}

\author{Alberto Enciso}
\address{Instituto de Ciencias Matem\'aticas, Consejo Superior de
  Investigaciones Cient\'\i ficas, 28049 Madrid, Spain}
\email{aenciso@icmat.es, dperalta@icmat.es}

\author{Daniel Peralta-Salas}
%\address{Instituto de Ciencias Matem\'aticas, Consejo Superior de
%  Investigaciones Cient\'\i ficas, 28049 Madrid, Spain}
%\email{dperalta@icmat.es}

%%    General info
%\subjclass[2010]{35B38, 58J05, 58K45}
%\date{\today}
%
%\keywords{ }
%
\begin{abstract}
In the context of magnetic fields generated by wires, we study the
connection between the topology of the wire and the topology of the
magnetic lines. We show that a generic knotted wire has
a magnetic line of the same knot type, but that given any pair of knots
there is a wire isotopic to the first knot having a magnetic line
isotopic to the second. These questions can be traced back to~Ulam in~1935.
\end{abstract}
\maketitle

\section{Introduction}

It is classical that the magnetic field generated by a closed wire is
given by the Biot--Savart law. That is, if $\ga(s)$ is a closed curve in
$\RR^3$ of length~$\ell$ and parametrized, say, by the arc-length parameter, its associated magnetic field is given by
the Biot--Savart integral
\[
B_\ga(x):= \int_0^\ell \frac{\dot\ga(s)\times (x-\ga(s))}{4\pi|x-\ga(s)|^3}\, ds\,.
\]
It is apparent that the vector field $B_\ga$, which is divergence-free, is in fact independent of the
parametrization of the curve, although it does depend on its
orientation.

While the magnetic fields created by simple curves, such as a circular
wire, are well understood and have been utilized to describe a wide
range of physical phenomena, our understanding of the magnetic fields
created by more complicated wires remains strikingly limited. This led
Ulam to consider the relationship between the degree of knottedness
of the wire and that of the associated periodic magnetic lines.

Specifically, Ulam posed the question of whether the magnetic field created by a knotted wire
must have a periodic magnetic line of nontrivial topology. This
is among the first problems of {\em The Scottish Book}~\cite[Problem
18]{Scottish}\/, the original manuscript of which goes back to
1935. This question of Ulam was set in a broader context in his later
problem collection~\cite[Section VII.7]{Ulam}, where he asks whether the
magnetic lines topologically reflect the knottedness of the wire.

Unlike many of the problems collected in {\em The Scottish Book}\/, the
progress made on these questions has been scarce. This is particularly
remarkable, on the one hand, because the existence of knotted fields
plays a significant role in a number of areas of physics (see
e.g.~\cite{Dennis,Irvine,Science}) and, on the other hand, because it
has been recently proposed to employ knotted wires to construct a new
magnetic confinement device, called the
knotatron~\cite{knotatron}. The reason for the lack of rigorous
contributions to this line of research is probably that, by the very
nature of the problems, the proof of these results must combine ideas
from the theory of dynamical systems, which are key to study the
knot types of the magnetic lines, with fine analytic
estimates for the Biot--Savart integral associated with a curve of
complicated geometry.

Our objective in this paper is to provide two somehow complementary
results concerning the relationship between the knottedness of the
wire and the topology of the magnetic lines. The first result asserts that
there are wires and magnetic lines whose knot types can be prescribed
at will (and independently). In particular, there are wires knotted in an
arbitrarily complicated way with an unknotted periodic magnetic line
and there are wires isotopic to the unknot with a magnetic line
isotopic to any given knot. More precisely, our first result can be stated
as follows:

\begin{theorem}\label{T1}
Let $\ga$ and $\tga$ be any two closed curves in $\RR^3$. Then
there is a wire isotopic to $\ga$ whose associated magnetic
field has a periodic magnetic line isotopic to~$\tga$.
\end{theorem}

Our second result asserts that, however, Ulam's original question can
be answered in the affirmative at least for $C^0$-generic curves. That
is, given any curve in the space, there is a smooth $C^0$-small deformation
of it which is isotopic to the original curve and has a magnetic line of the same
knot type. To state this result, we will say that two isotopic curves
are {\em arbitrarily $C^k$-close}\/ if the isotopy can
be taken arbitrarily close to the identity in the $C^k$~norm.

\begin{theorem}\label{T2}
Let $\ga$ be a closed curve in $\RR^3$. Given any integer $k$, there is a wire isotopic
to~$\ga$ and arbitrarily $C^0$-close to it whose associated magnetic
field has a periodic magnetic line isotopic to~$\ga$  and arbitrarily
$C^k$-close to it.
\end{theorem}

It is worth emphasizing that the magnetic lines constructed in the
above theorems are not just a mathematical curiosity, but they should
be observable in actual experiments. This is because they are
structurally stable, so any $C^1$~small perturbation of the magnetic
field (e.g., the magnetic field associated with a $C^1$~small
deformation of the wire) still has a magnetic line of the same
topology. In this direction, a feature of this magnetic line that we
find quite surprising is that it is in fact hyperbolic, not
elliptic, so in particular there are no toroidal magnetic surfaces in
a neighborhood of it.

The paper is organized as follows. In Section~\ref{S.asympt} we will
consider magnetic fields generated by a current density field defined
on a toroidal surface and compute the asymptotic behavior of the
corresponding Biot--Savart integral as the width of the surface tends
to zero. In Section~\ref{S.wires} we will connect this
kind of magnetic fields with fields created by wires. This will hinge
on a measure-theoretic argument that exploits the structure of
divergence-free vector fields on a torus whose integral curves are all
periodic. These results are put to use in Sections~\ref{S.T1}
and~\ref{S.T2}, where we respectively prove Theorems~\ref{T1} and~\ref{T2}.

We conclude the Introduction with a word about notation. Throughout
the paper, all the curves are assumed to be smooth and without
self-intersections unless stated otherwise, and the isotopies and
diffeomorphisms are always of class $C^\infty$. When it does not give
rise to confusion, we will use the same notation~$\ga$ for
a parametrized curve in space, which is a map $\ga(t)$ from the real line (or the circle) to~$\RR^3$,
and for its image $\ga\equiv \ga(\RR)\subset\RR^3$.

\section{Asymptotics for magnetic fields generated by
  surface currents}
\label{S.asympt}

A key ingredient in the proof of Theorems~\ref{T1} and~\ref{T2} will
be the use of
magnetic fields that are not generated by wires, but by current
densities supported on toroidal surfaces. Our goal in this section
is to analyze the behavior of magnetic fields of this
kind. We will be particularly interested in the case where the
toroidal surface is very thin.

To make things precise, let us consider the toroidal surface
\[
\surf:=\big\{ x\in \RR^3: \dist(x,\ga)=\ep\big\}\,,
\]
which is a smooth torus of the same knot type as the curve~$\ga$
provided that the width~$\ep$ is small enough. We will define
coordinates $(s,y)$ on the domain bounded by $\surf$ as follows.
Let $s$ be an arc-length parametrization of the curve~$\ga$, whose
length we will denote by $\ell$. This amounts to saying that the
tangent field~$\dot\ga$ has unit norm and $\ga(s)$ is
$\ell$-periodic, so in particular we can assume that
$s\in\es$. Without loss of generality, we will make the
assumption that the curve $\ga$ does not have any inflection points, which allows us to define the normal
and binormal vectors $\N(s),\B(s)$ at any point of the curve $\ga(s)$. It is well known that
this assumption is satisfied for
generic curves~\cite[p.~184]{Bruce} (roughly speaking, ``generic'' here
refers to an open and dense set, with respect to a reasonable $C^k$
topology, in the space of smooth curves in $\RR^3$).

Using the normal and binormal vector fields and denoting by $\DD^2$ the two-dimen\-sional
unit disk, we can introduce coordinates
$(s,y)\in(\es)\times\DD^2$ in the solid torus bounded by $\surf$ via the diffeomorphism
\begin{equation}\label{sy}
(s,y)\mapsto \ga(s)+\ep\ms y_1\ms \N(s)+\ep \ms y_2\ms \B(s)\,.
\end{equation}
Recall that the unit tangent vector is $\T(s):=\dot\ga(s)$.

In the coordinates $(s,y)$, a short computation using the Frenet
formulas yields the formula for the Jacobian of this change of
coordinates, which shows that the volume measure is written in these coordinates as
\begin{equation}\label{dV}
dx:=\ep^2\,[1-\ep\ka(s)\, y_1]\,ds\,dy\,.
\end{equation}
Here and in what follows, $\ka(s)$ and $\tau(s)$ respectively
denote the curvature and torsion of the curve~$\ga$. We will sometimes take polar coordinates $(r,\te)$  in the disk~$\DD^2$, which are
defined as
\[
y_1=r\cos\te\,,\quad y_2=r\sin\te\,.
\]
In terms of these coordinates, the volume reads as
\[
dx=\ep^2r\,[1-\ep\ka(s)\, r \cos\te]\,ds\,dr\,d\te\,.
\]

We shall next consider magnetic fields created by current densities
supported on~$\surf$. The current distribution will be given by $J\,
dS$, where $dS$ is the surface measure on the
surface~$\surf$ and the density $J$ is a smooth tangent vector field
defined on~$\surf$. The associated magnetic field is then
\[
\cB_J(x):= \int_{\surf} \frac{J(x')\times (x-x')}{4\pi|x-x'|^3}\, dS(x')\,.
\]
We will always assume that the divergence of~$J$ on the surface is zero (which is equivalent
to demanding that the divergence of the vector-valued measure $J\, dS$ is zero, in the sense
of distributions), so that $\cB_J$ indeed has the physical interpretation of a magnetic field. In particular, $\cB_J$ satisfies Maxwell's equations
\[
\curl \cB_J= J\, dS\,,\qquad \Div \cB_J=0\,.
\]

The main result of this section is the following lemma, which provides
the asymptotic behavior as~$\ep\to0$ of the magnetic field $\cB_J$ for
a wide family of current densities~$J$.  This family has two important
features: firstly, the fields are divergence-free, and secondly, they
have a special dependence on~$\ep$ that is crucial to prove the
existence of hyperbolic periodic magnetic lines isotopic
to~$\ga$. Since we will only be interested in the behavior of the
field near~$\ga$, let us restrict our attention to a fixed small neighborhood of the core curve~$\ga$ (for instance, the
region $|y|<\frac1{10}$). In this region, if $h(\ep)$ is a
function that tends to zero as $\ep\to0$ (typically~$\ep^m$ or~$\ep\log\ep$), let us agree to say that a scalar quantity $q(s,y,\ep)$ is of order
$O(h(\ep))$ if
\begin{subequations}\label{defO}
\begin{equation}\label{defO1}
|\pd_s^jD_y^k q(s,y,\ep)|\leq C_{j,k}\,h(\ep)
\end{equation}
for all $j,k$. Likewise, given a
nonnegative integer~$m$ we will say
that $q(s,y,\ep)$ is of order $O(|y|^m)$ if
\begin{equation}\label{defO2}
|\pd_s^jD_y^k q(s,y,\ep)|\leq C_{j,k}|y|^{\max\{m-k,0\}}
\end{equation}
in the above region for all $j,k$ and uniformly in $\ep$. We will use the notations 
\begin{equation}\label{defO3}
O(1)\equiv
O(\ep^0)\quad \text{and} \quad O(h(\ep)+|y|^k):=O(h(\ep))+O(|y|^k)\,.
\end{equation}
\end{subequations}

\begin{lemma}\label{L.asympt}
Let us take the divergence-free tangent vector field 
\begin{equation}\label{defJ}
J :=\frac1{1-\ep\ka(s)\, \cos\te}\bigg(F(\te)\,\pd_s+\frac{G(s)}{\ep}\,\pd_\theta\bigg)\,,
\end{equation}
defined on $\surf$ in terms of two smooth functions $F$ and~$G$ of
period~$2\pi$ and $\ell$, respectively. We will also assume that $G$ does
not vanish. Consider the current distribution on $\surf$
given by $J\, dS$.
Then, using the order notation~\eqref{defO},
\begin{multline*}
\cB_J=\big[G(s)+O(\ep\log\ep+|y|)\big]\,
\pd_s+\frac{b_1+b_2y_1-a_2y_2+O(\ep\log\ep+|y|^2)}{2\ep}\,\pd_{y_1} \\
-\frac{a_1+a_2y_1+b_2y_2+O(\ep\log\ep+|y|^2)}{2\ep}\,\pd_{y_2}\,,
\end{multline*}
where $a_n,b_n$ are the Fourier coefficients of $F$:
\[
F(\te)=:\sum_{n=0}^\infty(a_n \cos n\te+ b_n\sin n\te)\,.
\]
\end{lemma}

\begin{proof}
Let us start by recalling that Equation~\eqref{sy} and the Frenet formulas
imply that the vector fields $\pd_s$, $\pd_r$ and $\pd_\te$
corresponding to the coordinates $(s,r,\te)$ can be written in
terms of the curvature~$\ka$ and torsion~$\tau$ of the curve $\ga$ as:
\begin{align}
\pd_s&\equiv\frac{\pd x}{\pd s}= (1-\ep \ka y_1)\T+\ep \tau
y_2\N - \ep\tau y_1 \B\,.\label{pds}\\
\pd_r&\equiv \frac{\pd x}{\pd r}=\frac{\ep y_1\N+\ep y_2 \B}{|y|}\,,\label{pdr}\\
\pd_\te&\equiv \frac{\pd x}{\pd \te}= -\ep y_2\N + \ep y_1\B\,.\label{pdte}
\end{align}
Here and in what follows, the curvature, torsion and elements of the
Frenet basis are evaluated at~$s$ unless stated otherwise. 

As the vector field $\pd_r$
has norm $\ep$, it follows from the expression~\eqref{dV} for the volume in the
coordinates $(s,r,\te)$ that the surface measure of $\surf$ can be
written as
\[
dS=\ep A\, ds\,d\te\,,
\]
with $A(s,\te):= 1-\ep\ka(s) \cos\te$. Hence
the field~$J$ is indeed divergence-free on~$\surf$, since the usual
formula for the divergence yields
\begin{align*}
\Div_{\surf} J=
\frac1{A(s,\te)}\frac\pd{\pd s}\bigg(\frac{A(s,\te) \,F(\te)}{1-\ep\ka(s)\,
  \cos\te}\bigg) +
\frac1{A(s,\te)}\frac\pd{\pd \te}\bigg(\frac{A(s,\te) \, G(s)}{\ep(1-\ep\ka(s)\,
  \cos\te)}\bigg) =0\,.
\end{align*}

Let us fix a point $x$ in a small neighborhood of the curve~$\ga$ that
we will describe by its coordinates $(s,y)\in(\es)\times\DD^2$, which will remain fixed
in the argument. Let us fix a small number $\de$ and write the
equation for $\cB_J$ as
\begin{align}\label{BI}
\cB_J(x)&=\int_\surf \frac{J(x')\times (x-x')}{4\pi|x-x'|^3}\, dx'=I+ O(\ep)\,,
\end{align}
where
\[
I=\ep \int_0^{2\pi}\int_{s-\de}^{s+\de}\frac{J(x')\times
  (x-x')}{4\pi|x-x'|^3}\, A(s',\te')\, ds'\,d\te'
\]
and $(s',\te')$ are the coordinates used to parametrize the point
$x'\in\surf$.

To see why this is true, it suffices
to observe that $J$~is uniformly bounded in~$\ep$ (here one has to use
that the norm of $\frac1\ep\pd_\te$ is~1 on the surface) and
that the distance $|x-x'|$ is at least $\de$ when $|s-s'|>\de$. This
readily yields
\[
|\cB_J(x)-I|=\bigg|\ep\int_0^{2\pi}\int_{|s'-s|>\de}\frac{J(x')\times
  (x-x')}{4\pi|x-x'|^3}\, A(s',\te')\, ds'\,d\te'\bigg|\leq C_\de |\surf|=O(\ep)\,.
\]
An analogous reasoning can obviously be applied to the derivatives of
$\cB_J(x)-I$, so Equation~\eqref{BI} follows.

We are interested in the asymptotic behavior of $I$ for small
$\ep$. In order to compute it, we will need to expand the quantity
$x-x'$ in the variables~$\ep$ and $s'-s$, which can be achieved
using Equation~\eqref{sy} and the Frenet formulas:
\begin{equation*}
x-x'=-(s'-s)\T+\ep(y_1-\cos\te')\N+\ep (y_2-\sin\te')\B + \cO(2)\,.
\end{equation*}
Hereafter we will use the notation $\cO(k)$ to denote terms in the
Taylor expansion that are at least a $k^{\text{th}}$ power in
$(\ep,s'-s)$. For example,
\[
(s'-s)^{2j}\,,\quad \ep(s'-s)^j\,,\quad \ep^{1+j}
\]
are all $\cO(2)$ if $j\geq1$. Likewise, using the formulas~\eqref{pds}--\eqref{pdte} the vector field $J(x')$ can be written as
\begin{equation*}
J(x')= F(\te')\,\T + G(s)\,\big(\cos\te'\B-\sin\te'\N\big)+ \cO(1)\,,
\end{equation*}
where we recall that $\T$, $\N$ and $\B$ are evaluated at~$s$.

Since the orthonormal basis $\{\T,\N,\B\}$ is positively oriented, it
then follows that
\begin{multline}
J(x')\times (x-x')= \ep \, G(s)\, \big(1-y_1\cos\te'-y_2\sin\te'\big)\T
\\-\Big[(s'-s)\,G(s)\,\cos\te'+\ep (y_2-\sin\te')F(\te')\Big]\N
\\-\Big[(s'-s)G(s)\sin\te'-\ep (y_1-\cos\te')F(\te')\Big]\B +\cO(2)\label{Jx}
\end{multline}
and
\begin{equation}\label{square}
|x-x'|^2=(s'-s)^2+\ep^2(1 +\be)+\cO(3)\,,
\end{equation}
with 
\[
\be:= -2y_1\cos\te'-2y_2\sin\te' +|y|^2\,.
\]

As the surface measure is 
\[
dS(x')=\ep(1+\cO(1))\, ds'\, d\te'\,,
\]
one can now use the formulas~\eqref{Jx} and~\eqref{square} to write
$I$ as 
\begin{align}\label{eqI}
I=\frac\ep{4\pi}\int_0^{2\pi}\int_{s-\de}^{s+\de} \frac{h_1 \T+ h_2\N+
  h_3\B}{[(s'-s)^2+\ep^2(1+\be)+\cO(3)]^{3/2}}\, ds'\, d\te'\,.
\end{align}
Here we have set
\begin{align*}
h_1&:=\ep\,G(s)\,\big(1-y_1\cos\te'-y_2\sin\te'\big)+ \cO(2)\,,\\
h_2&:= - \Big[(s'-s)G(s)\cos\te'+\ep (y_2-\sin\te')F(\te')\Big]+ \cO(2)\,,\\
h_3&:=- \Big[(s'-s)G(s)\sin\te'-\ep (y_1-\cos\te')F(\te')\Big]+ \cO(2)\,.
\end{align*}

Let us begin with the tangent component of $I$,
\begin{align*}
I_1&:=\frac\ep{4\pi}\int_0^{2\pi}\int_{s-\de}^{s+\de}
\frac{h_1}{[(s'-s)^2+\ep^2(1+\be)+\cO(3)]^{3/2}}\, ds'\, d\te'\,.
\end{align*}
Using again that $1+\be$ is positive for small enough~$y$, clearly
\[
(s'-s)^2+\ep^2(1+\be)+\cO(3)=\Big((s'-s)^2+\ep^2(1+\be)\Big)\, \big(1+\cO(1)\big)\,,
\]
so one can decompose $I_1$ as
\begin{align}
I_1&=\frac\ep{4\pi}\int_0^{2\pi}\int_{s-\de}^{s+\de}
\frac{\ep\,G(s)\,\big(1-y_1\cos\te'-y_2\sin\te'\big)}{[(s'-s)^2+\ep^2(1+\be)]^{3/2}}\,
ds'\, d\te'\notag\\
&\qquad\qquad + \ep\int_0^{2\pi}\int_{s-\de}^{s+\de}
\frac{\cO(2)}{[(s'-s)^2+\ep^2(1+\be)]^{3/2}}\, ds'\, d\te'\notag\\[1mm]
&=I_{11}+I_{12}\,. \label{defI1}
\end{align}
Since $1+\be>0$ for small~$y$, the second integral can be easily bounded as
\begin{align*}
|I_{12}|&\leq C\ep
\int_0^{2\pi}\int_{s-\de}^{s+\de}\frac{(s'-s)^2+\ep|s-s'|+\ep^2}{[(s'-s)^2+\ep^2(1+\be)]^{3/2}}\,
ds'\, d\te'\\
&\leq C\ep \int_{s-\de}^{s+\de}
\frac{1}{[(s'-s)^2+\ep^2]^{1/2}}\, ds'\\
&\leq C\ep\,\log\frac1\ep\,.
\end{align*}
To study the integral $I_{11}$ in~\eqref{defI1}, let us introduce the
variable 
\[
t:=\frac{s'-s}{\ep (1+\be)^{\frac12}}\,,
\]
in terms of which the integral reads as
\begin{align*}
I_{11}&=\frac1{4\pi}\int_0^{2\pi}\int_{-\de/[\ep (1+\be)^{\frac12}]}^{\de/[\ep (1+\be)^{\frac12}]}
\frac{G(s)\,(1-y_1\cos\te'-y_2\sin\te')}{(1+\be)\, (t^2+1)^{3/2}}\,
dt\, d\te'\\[1mm]
&=\frac1{4\pi}\int_0^{2\pi}\int_{-\infty}^{\infty}
\frac{G(s)\,(1-y_1\cos\te'-y_2\sin\te')}{(1+\be)\,(t^2+1)^{3/2}}\, dt\, d\te'+O(\ep^2)\\[1mm]
&=\frac{G(s)}{2\pi} \int_0^{2\pi}
\big(1-y_1\cos\te'-y_2\sin\te'\big)(1+2y_1\cos\te'+2y_2\sin\te') \, d\te'
\\
&\qquad \qquad \qquad \qquad \qquad \qquad \qquad \qquad \qquad \qquad \qquad \qquad \qquad+O(\ep^2+|y|^2)\\
&=G(s)+O(\ep^2+|y|)\,.
\end{align*}
To pass to the third line we have expanded in~$y$ and used that
\[
\int_{-\infty}^\infty \frac{dt}{(t^2+1)^{3/2}}=2\,.
\]

Now that we are done with the computation of $I_1$, let us consider
next the normal component of $I$ and decompose it as before:
\begin{align*}
I_2&:=\frac\ep{4\pi}\int_0^{2\pi}\int_{s-\de}^{s+\de} \frac{
  h_2}{[(s'-s)^2+\ep^2(1+\be)]^{3/2}}\, ds'\, d\te'\\
&=-\frac\ep{4\pi}\int_0^{2\pi}\int_{s-\de}^{s+\de} \frac{
   (s'-s)G(s)\cos\te'}{[(s'-s)^2+\ep^2(1+\be)]^{3/2}}\, ds'\,
 d\te'\\
&\qquad\qquad\qquad\qquad - \frac{\ep^2}{4\pi}\int_0^{2\pi}\int_{s-\de}^{s+\de} \frac{
   (y_2-\sin\te')F(\te')}{[(s'-s)^2+\ep^2(1+\be)]^{3/2}}\, ds'\, d\te' \\
&\qquad\qquad\qquad \qquad\qquad\qquad \qquad+ \ep\int_0^{2\pi}\int_{s-\de}^{s+\de} \frac{
      \cO(2)}{[(s'-s)^2+\ep^2(1+\be)]^{3/2}}\, ds'\, d\te'\\
&=:I_{21}+I_{22}+I_{23}\,.
\end{align*}
Arguing as in the case of $I_1$, one immediately gets that
\[
|I_{23}|\leq C\ep\log\frac1\ep\,,
\]
and the fact that the integrand is an odd function of $s'-s$
immediately implies that
\[
I_{21}=0\,.
\]
Moreover, the integral $I_{22}$ can be analyzed just as in the case of~$I_{11}$,
yielding
\[
I_{22}=\frac{b_1+b_2y_1-a_2y_2}2+O(\ep+|y|^2)\,.
\]

The binormal component of $I$, 
\[
I_3:=\frac\ep{4\pi}\int_0^{2\pi}\int_{s-\de}^{s+\de} \frac{ h_3}{[(s'-s)^2+\ep^2(1+\be)+\cO(3)]^{3/2}}\, ds'\, d\te'\,,
\]
can be computed as in the case of~$I_2$, obtaining
\[
I_3=-\frac{a_1+a_2y_1+b_2y_2}2+O(\ep+|y|^2)\,.
\]
As Equations~\eqref{pds}--\eqref{pdte} obviously imply that
\[
\T=(1+O(\ep))\pd_s+ O(|y|)\, \pd_{y_1}+ O(|y|)\, \pd_{y_2}\,,\quad
\N=\frac1\ep\pd_{y_1}\,,\quad \B=\frac 1\ep\pd_{y_2}\,,
\]
one obtains the desired asymptotic formula for $\cB_J$.

It is clear that the same method yields formulas for the derivatives of the components
of $\cB_J$, which correspond to the derivatives of the terms that we
have already computed (e.g., in the case of first order derivatives,
$\pd_s I_k$ and $\pd_{y_j}I_k$). To illustrate the reasoning, let us
consider $\pd_s I_1$. Since the point of coordinates
$(s,y)$ is in the interior of the solid torus bounded by~$\surf$, 
one can safely differentiate under the integral sign to find:
\begin{multline*}
\pd_s I_1=\frac\ep{4\pi}\int_0^{2\pi}\int_{s-\de}^{s+\de}
\frac\pd{\pd s}\bigg(\frac{\ep\,G(s)\,\big(1-y_1\cos\te'-y_2\sin\te'\big)}{[(s'-s)^2+\ep^2(1+\be)]^{3/2}}\bigg)\,
ds'\, d\te'\\
+ \ep\int_0^{2\pi}\int_{s-\de}^{s+\de}
\frac\pd{\pd
  s}\bigg(\frac{\cO(2)}{[(s'-s)^2+\ep^2(1+\be)]^{3/2}}\bigg)\, ds'\,
d\te'\,.
\end{multline*}
Here we have used that the boundary
terms cancel out by parity. Using that for all $q$
of order $\cO(2)$ one can write
\[
\pd_sq=\pd_{s'}q_1+q_2
\]
with $q_j$ also of order $\cO(2)$,
we can further simplify
these integrals as:
\begin{align*}
\pd_s I_1&=\frac\ep{4\pi}\int_0^{2\pi}\int_{s-\de}^{s+\de}
\frac{\ep\,G'(s)\,\big(1-y_1\cos\te'-y_2\sin\te'\big)+\cO(2)}{[(s'-s)^2+\ep^2(1+\be)]^{3/2}}\,
ds'\, d\te'\\
&\qquad \qquad - \ep\int_0^{2\pi}\int_{s-\de}^{s+\de}
\frac\pd{\pd
  s'}\bigg(\frac{\ep\,G(s)\,\big(1-y_1\cos\te'-y_2\sin\te'\big)+\cO(2)}{[(s'-s)^2+\ep^2(1+\be)]^{3/2}}\bigg)\, ds'\,
d\te'\\
&=\frac\ep{4\pi}\int_0^{2\pi}\int_{s-\de}^{s+\de}
\frac{\ep\,G'(s)\,\big(1-y_1\cos\te'-y_2\sin\te'\big)+\cO(2)}{[(s'-s)^2+\ep^2(1+\be)]^{3/2}}\,
ds'\, d\te'\,.
\end{align*}
Here we are using a parity argument both to get rid of the boundary
terms that appear when one integrates by parts and to neglect the terms of $\cO(2)$ that are odd functions of
$s'-s$, as they do not contribute to the integral.
As the above integral is of the same form as $I_1$, the previous reasoning
immediately yields
\[
\pd_s I_1=G'(s)+O(\ep\log\ep+|y|)\,.
\]
The derivatives with respect to~$y$, which are in fact easier, can be
handled with a completely analogous argument.
\end{proof}

\begin{remark}
The norm of the field $\cB_J$ is of order $1$, and the reason for which
the components of $\cB_J$ along the fields $\pd_{y_j}$ are of order
$1/\ep$ is simply that the norm $|\pd_{y_j}|$ is~$\ep$ (recall that
they are simply the normal or binormal vector divided by~$\ep$). 
\end{remark}

\section{From surface currents to closed wires}
\label{S.wires}

In this section we will derive tools that permit us to show that there
are configurations of wires that create magnetic fields which
approximate, in a certain sense, magnetic fields generated by current
densities of the form studied in Section~\ref{S.asympt}. For
simplicity, throughout this section we will denote by~$\cS$ a surface
of $\RR^3$ diffeomorphic to a torus.

An important ingredient in the proof will be the idea of convergence
of measures. We recall that a sequence of vector-valued measures
$d\La_n$ supported on~$\cS$ converges weakly to $d\La$ if, given any
continuous function $u:\cS\to\RR^3$ one has
\[
\lim_{n\to\infty}\int_\cS u\cdot d\La_n= \int_\cS u\cdot d\La\,.
\]
In this direction, an easy but very useful result is the following:

\begin{lemma}\label{L.conv}
Let $K$ be a compact subset of $\RR^3$. Consider a sequence of
vector-valued measures $d\La_n$ whose supports are contained in a
compact set $K'$ and assume that this sequence converges weakly to $d\La$. If
$K$ does not intersect~$K'$, then
\begin{equation}\label{estconvL}
\lim_{n\to\infty}\bigg\|\int_{K'}\frac{d\La_n(x')\times
  (x-x')}{4\pi|x-x'|^3}-\int_{K'}\frac{d\La(x')\times
  (x-x')}{4\pi|x-x'|^3}\bigg\|_{C^m(K)}=0
\end{equation}
for any integer~$m$.
\end{lemma}

\begin{proof}
Observe that the kernel
\begin{equation}\label{Kxxp}
(x,x')\mapsto \frac{x-x'}{4\pi|x-x'|^3}
\end{equation}
is continuous in the set $(x,x')\in K\times K'$. The
convergence of the measures $d\La_n$ to $d\La$ and the fact that these measures are supported
on $K'$ imply that, uniformly for all $x\in K$, 
\[
\lim_{n\to\infty} \int_{K'} \frac{d\La_n(x')\times
  (x-x')}{4\pi|x-x'|^3} = \int_{K'} \frac{d\La(x')\times
  (x-x')}{4\pi|x-x'|^3}\,.
\]
Since the derivatives of the kernel~\eqref{Kxxp} with respect to~$x$
are also continuous on~$K\times K'$, the same argument yields the
$C^k$~convergence~\eqref{estconvL} on the set~$K$.
\end{proof}

The next lemma shows how to approximate the magnetic field created by a surface
distribution $J\, dS$ through the Biot--Savart integral (cf.\
Equation~\eqref{BI}) by that of a collection of magnetic
wires. Concerning the statement of the lemma, it is
worth noting that we require the tangent vector field~$J$ to be
divergence-free. This is not an additional condition that we impose
because we want to interpret this field as a current density, but an
actual necessary technical condition for the statement of the lemma to
hold true. This is because the divergence of any measure of the form
$\dot\ga_k(t)\, dt$ can be easily
seen to be zero, in the sense of distributions, so the fact that there is a collection of measures $d\La_n$ of
the form~\eqref{dLa1} that converge to $J\, dS$ automatically
implies that the tangent field $J$ is divergence-free on~$\cS$ (or,
equivalent, that the divergence of $J\, dS$ is zero).

\begin{lemma}\label{L.from}
 Let $J$ be a tangent vector field on the toroidal surface~$\cS$ whose
 divergence on the surface is zero. Let us assume that $J$ does not
 vanish and that all its integral curves are periodic. Then there exist a 
 positive constant $c_0$ and a sequence of finite collections of
 (distinct) periodic integral curves of~$J$,
$\ga_k:(\RR/T_k\ZZ)\to\cS$, such that the vector-valued measure
\begin{equation}\label{dLa1}
d\La_n:=\frac{c_0}n\sum_{k=1}^n \dot\ga_k(t)\, dt
\end{equation}
converges weakly to $J\, dS$ as $n\to\infty$. Here $T_k$ is
the minimal period of the integral curve $\ga_k$.
\end{lemma}

\begin{proof}
It is known (cf.\ e.g.~\cite[4.1.14]{Godbillon}) that if $J$ is a non-vanishing
divergence-free field on the torus whose integral curves are all
periodic, then:
\begin{enumerate}
\item There is a closed transverse curve~$\cC$ on~$\cS$ which
  intersects all the integral curves of~$J$ at exactly one point.

\item The period function $T:\cS\to(0,\infty)$, which maps each point
  on the surface to the minimal period of the integral curve of~$J$ passing
  through it, is smooth.
\end{enumerate}

Let us now define the isochronous field associated to~$J$,
\begin{equation}\label{tJ}
\tJ:=T\, J\,,
\end{equation}
whose integral curves are all closed and of
period~1. Consider a periodic coordinate on~$\cC$ of period~1,
\begin{equation}\label{Theta}
\Theta: \cC\to\TT^1\,,
\end{equation}
with $\TT^1:= \RR/\ZZ$. Let us now construct a
map $\Psi: \cS\to \TT^2$ by setting
\begin{equation}\label{defPsi}
\Psi^{-1}(\al,\si):=\phi_\si[\Theta^{-1}(\al)]\,,
\end{equation}
where $\phi_\si$ is the flow at time~$\si$ of the field~$\tJ$. Since~$\cC$
intersects each integral curve exactly once and all integral curves of ~$\tJ$
have period~1, it is obvious that $\Psi$ is a
diffeomorphism. Moreover, it is apparent
that one can write the push-forward of~$\tJ$ under the
diffeomorphism~$\Psi$ as
\[
\Psi_*\tJ=\pd_\si\,.
\]

As the period function takes the
same value on each integral curve of~$J$, there is a smooth function
$\tT:\TT^1\to(0,\infty)$ such that
\begin{equation}\label{tT}
T[\Psi^{-1}(\al,\si)]=\tT(\al)\,.
\end{equation}
In particular, $T$ is a first integral of~$J$, so it is trivial that the
divergence of~$\tJ$ on the surface is also zero. Hence, it
follows that the push-forward $\tsi$ of the area 2-form on~$\cS$ to~$\TT^2$ can be written as
\begin{equation}\label{B}
\tsi= B(\al)\, d\al\wedge d\si\,,
\end{equation}
where $B:\TT^1\to\RR$ is a positive smooth function. In order to see
this, it suffices to write $\tsi= B\, d\al\wedge d\si$ and notice that the field $\Psi_*\tJ=\pd_\si$ is divergence-free
with respect to~$\tsi$, which implies that
\[
0=(\Div_{\tsi} \pd_\si)\, \tsi = d (i_{\pd_\si} \tsi)= -d(B\, d\al)\,,
\]
which shows that $B$ is independent of~$\si$.

Let us consider a sequence of points $\al_k\in \TT^1$ that is
uniformly distributed with respect to the probability measure on~$\TT^1$ given by 
\begin{equation}\label{measure}
d\rho:=\frac{B(\al)}{c_0 \, \tT(\al)}\, d\al\,,
\end{equation}
where $B(\al)$ is the positive function defined in~\eqref{B} and
\[
c_0:=\int_{\TT^1} \frac{B(\al)}{\tT(\al)}\, d\al\,.
\]
Specifically, this means that, for any interval $I$ of $\TT^1$,
\begin{equation}\label{unif}
\lim_{n\to\infty}\frac{\text{\#}\{k\leq n: \al_k\in I\}}n =\int_I d\rho\,.
\end{equation}

Let $\tga_k:\TT^1\to\cS$ be the integral curve of~$\tJ$ with initial
condition $\tga_k(0)=\Theta^{-1}(\al_k)$.
We shall next check that 
\begin{align*}
d\La_n:=\frac{c_0}n\sum_{k=1}^n \dot\tga_k(\si)\, d\si
\end{align*}
converges weakly to $J\, dS$. Before doing it, let us observe that
this implies the lemma, because if $\ga_k(t)$ denotes the integral
curve of~$J$ with initial condition $\ga_k(0)=\Theta^{-1}(\al_k)$, it
is obvious from the invariance of the measure under reparametrization
(namely, the fact that $\dot\tga_k(\si)\, d\si =\dot \ga_k(t)\, dt$) that $d\La_n$ is also given by the formula~\eqref{dLa1}
provided in the statement.

To show that $d\La_n$ converges weakly to $J\, dS$, recall that the
fact that the points $\{\al_k\}$ are uniformly distributed with
respect to the probability measure $d\rho$ is equivalent to saying
that the measure
\[
d\rho_n:=\frac1n\sum_{k=1}^n \de_{\al_k}
\]
on $\TT^1$ converges weakly to $d\rho$ as $n\to\infty$. Hence,
\begin{align*}
  \int u\cdot d\La_n&=\frac{c_0}n\sum_{k=1}^n\int_{\TT^1}\langle
  \Psi_*
  u,\pd_\si\rangle|_{(\al_k,\si)}\, d\si\\
  &=c_0\int_{\TT^2}\langle \Psi_* u,\pd_\si\rangle\, d\rho_n\, d\si
\end{align*}
satisfies
\begin{align*}
\lim_{n\to\infty}  \int u\cdot d\La_n&=c_0\int_{\TT^2}\langle \Psi_*
u,\pd_\si\rangle\, d\rho\, d\si\\
&=\int_{\TT^2}\langle \Psi_* u,\pd_\si\rangle\,
\frac{B(\al)}{\widetilde T(\al)}\, d\al\, d\si\\
&=\int_\cS u\cdot J\, dS\,,
\end{align*}
where we have used that $\tJ= T\, J$.
\end{proof}

In the proof of Theorem~\ref{T2} we will need the following lemma,
which is basically a refinement of Lemma~\ref{L.from} that provides a
versatile sufficient condition for a sequence of vector-valued
measures supported on curves (not necessarily integral curves of the
field~$J$) to converge to the current distribution $J\, dS$:

\begin{lemma}\label{L.curve}
Let $J$ be as in Lemma~\ref{L.from} and consider the associated map
$\Theta:\cC\to\TT^1$ defined in~\eqref{Theta}. Suppose that $\Ga_n :
\RR/T_n\ZZ\to\cS$ is a sequence of periodic curves without
self-intersections that satisfy the
following properties, with $T_n$ being the minimal period:
\begin{enumerate}
\item The curves $\Ga_n$ intersect~$\cC$ transversally and their
  points of intersection are uniformly distributed with
  respect to the measure $d\rho$ defined in~\eqref{measure}, that is,
  for any open subset $I\subset \cC$ one has
\[
\lim_{n\to\infty}\frac{\text{\#}(\Ga_n\cap I)} {\text{\#}(\Ga_n\cap \cC)}=\int_{\Theta(I)}d\rho\,.
\]

\item The tangent vectors $\dot\Ga_n(\si)$ converge uniformly to the
  isochronized field~$\tJ$, defined in~\eqref{tJ}:
\[
\lim_{n\to\infty}\sup_{\si\in\RR}\big|\dot\Ga_n(\si) - \tJ(\Ga_n(\si))\big|=0\,.
\]
\end{enumerate}
Then there is a positive constant $c_0$ such that the vector-valued measures
\[
d\la_n:=\frac{c_0}{\text{\#}(\Ga_n\cap \cC)}\, \dot\Ga_n(\si)\, d\si \,,
\]
which are supported on $\Ga_n$, converge weakly to $J\, dS$ as $n\to\infty$.
\end{lemma}

\begin{proof}
We will use the notation introduced in the proof of Lemma~\ref{L.from}
without further notice. Setting 
\[
N_n:=\text{\#}(\Ga_n\cap \cC)\,,
\]
let us denote the intersection points of $\Ga_n$ with the transverse
curve~$\cC$ by
\[
\Ga_n\cap \cC=\{p_{n,k}\}_{k=1}^{N_n}\,,
\]
where we are labeling the points~$p_{n,k}$ so that they correspond to
consecutive intersection points. We will denote the intersection times
of the curve~$\Ga_n$ by~$\si_{n,k}\in [0,T_n)$, with $\si_{n,1}=0$ and
\[
p_{n,k}=\Ga_n(\si_{n,k})\,.
\]
The point in the unit circle associated to~$p_{n,k}$ under the
map~$\Theta$ will be denoted by $\al_{n,k}:=\Theta(p_{n,k})$.

The uniform convergence of $\dot\Ga_n(\si)$ to
$\tJ(\Ga_n(\si))$ and the fact that all integral curves of $\tJ$ are
closed with period~1 obviously imply that the time between
consecutive intersections with~$\cC$ tends to~1:
\begin{equation}\label{eqnueva}
\lim_{n\to\infty} \max_{1\leq k\leq N_n} |\si_{n,k+1}-\si_{n,k}-1|=0\,.
\end{equation}
Throughout we are identifying $\si_{n,N_n+1}:=T_n$. In particular,
this implies that $T_n/N_n$
tends to~1. Letting
$\tga_{n,k}:\TT^1\to\cS$ be the integral curve of the
isochronized field~$\tJ$ with initial condition
$\tga_{n,k}(0)=p_{n,k}$, we then infer that the integral curve~$\tga_{n,k}$ is
close to $\Ga_n$ in the sense that
\begin{equation}\label{convC1}
\lim_{n\to\infty}\max_{1\leq k\leq N_n} \big\|\Ga_n(\cdot)
-\tga_{n,k}(\,\cdot- \si_{n,k})\big\|_{C^1((\si_{n,k},\si_{n,k+1}))}=0\,.
\end{equation}

Let us define $c_0$ as in the proof of Lemma~\ref{L.from}. To prove that 
\[
d\la_n:=\frac{c_0}{N_n}\, \dot\Ga_n(\si)\, d\si 
\]
converges weakly to $J\, dS$, let us take an arbitrary smooth
function $u:\cS\to\RR^3$, which without loss of generality can be
thought of as a tangent vector field on~$\cS$. Notice that
\begin{multline*}
\cE_n:=\max_{1\leq k\leq N_n} \bigg|\int_{\si_{n,k+1}-\si_{n,k}}^1 u(\tga_{n,k}(\si))\cdot
\dot\tga_{n,k}(\si)\, d\si \bigg|\\
+ \max_{1\leq k\leq N_n}\int_{\si_{n,k}}^{\si_{n,k+1}} \big|
u(\Ga_n(\si))\cdot \dot\Ga_n(\si) - u(\tga_{n,k}(\si-\si_{n,k})) \cdot\dot
\tga_{n,k}(\si-\si_{n,k})\big|\, d\si 
\end{multline*}
tends to zero as $n\to\infty$ by~\eqref{eqnueva}
and~\eqref{convC1}. Setting
\[
d\La_n:=\frac{c_0}{N_n}\sum_{k=1}^{N_n}\dot\tga_{n,k}(\si)\, d\si\,,
\]
one then has that
\begin{align*}
\bigg|\int_{\cS} u&\cdot d\la_n- \int_{\cS} u\cdot d\La_n\bigg|\leq
\frac{c_0}{N_n}\sum_{k=1}^{N_n} \bigg|\int_{\si_{n,k+1}-\si_{n,k}}^1 u(\tga_{n,k}(\si))\cdot
\dot\tga_{n,k}(\si)\, d\si \bigg|\\
&+ \frac{c_0}{N_n}\sum_{k=1}^{N_n} \int_{\si_{n,k}}^{\si_{n,k+1}} \big|
u(\Ga_n(\si))\cdot \dot\Ga_n(\si) - u(\tga_{n,k}(\si-\si_{n,k})) \cdot\dot
\tga_{n,k}(\si-\si_{n,k})\big|\, d\si \\
&\leq c_0\, \cE_n
\end{align*}
also tends to zero as $n\to\infty$. Here we have used that
\begin{multline*}
\int_{\si_{n,k}}^{\si_{n,k}+1} u(\tga_{n,k}(\si))\cdot
\dot\tga_{n,k}(\si)\, d\si= \int_{\si_{n,k}}^{\si_{n,k+1}} u(\tga_{n,k}(\si))\cdot
\dot\tga_{n,k}(\si)\, d\si\\
+ \int_{\si_{n,k+1}-\si_{n,k}}^1 u(\tga_{n,k}(\si))\cdot
\dot\tga_{n,k}(\si)\, d\si\,.
\end{multline*}
As the points $\al_{n,k}$ are equidistributed with respect to the
probability measure~$d\rho$, it follows directly from the proof of Lemma~\ref{L.from} that the
measures $d\La_n$ converge to $J\, dS$ as $n\to\infty$. The lemma
is then proved.
\end{proof}

\begin{remark}
Using the coordinates $(\al,\si)$ introduced in the proof of
Lemma~\ref{L.from}, one can visually understand these results as
follows. In these coordinates on~$\cS$ (which do not generally arise from a
diffeomorphism of the ambient space $\RR^3$, as they change the knot
type of the integral curves of~$J$), the integral curves of $J$ are
the vertical circles $\{\al=\text{constant}\}$. The transverse curve $\cC$
corresponds to the equatorial circle $\{\si=0\}$ and the collection of integral curves $\ga_k$
constructed in Lemma~\ref{L.from} is simply a collection of vertical
circles $\{\al=\al_k\}$ with $\al_k$ distributed according to certain
probability measure. The curves $\Ga_n$ satisfying the assumptions of
Lemma~\ref{L.curve} correspond to nearly vertical periodic curves
which close after winding once in the horizontal direction and $N_n$
times in the vertical direction.
\end{remark}

\section{Magnetic lines and wires of arbitrary topology}
\label{S.T1}

In this section we will prove Theorem~\ref{T1}. For the sake of
clarity, let us divide the argument in several steps:

\subsubsection*{Step 1: Construction of a surface current distribution with a
  hyperbolic magnetic line isotopic to~$\tga$}

Let us consider the toroidal surface $\sur$ of core curve~$\tga$ and small
width~$\ep$, and the divergence-free tangent vector field~$J$ on~$\sur$ given by
\begin{equation}\label{defJ2}
J:=\frac1{1-\ep\ka(s)\, \cos\te}\,J_0\,,\qquad J_0:= 2\cos2\te\,
\pd_s+ \frac1\ep\,\pd_\te\,.
\end{equation}
Notice that the field~$J$ is of the form~\eqref{defJ}, so
Lemma~\ref{L.asympt} ensures that the magnetic
field $\cB_J$ generated by the surface current distribution $J\, dS$
is given by
\[
\cB_J= \big[1+q_0(\ep,y)\big]\, \pd_s -\frac{y_2+q_1(\ep,y)}\ep\pd_{y_1}-\frac{y_1+q_2(\ep,y)}\ep\pd_{y_2}\,.
\]
where
\[
q_0(\ep,y)= O(\ep\log\ep+|y|)\,,\quad q_1(\ep,y)=O(\ep\log\ep+|y|^2)\,,\quad q_2(\ep,y)=O(\ep\log\ep+|y|^2)\,.
\]

Let us consider the vector field on the domain~$\cT$ bounded by~$\sur$ given by
\[
X:=-(y_2+q_1(0,y))\,\pd_{y_1}-(y_1+q_2(0,y))\,\pd_{y_2}\,,
\]
which vanishes identically on the curve $\tga\equiv \{y=0\}$ (of
course, for $\ep=0$ we are taking $\ep\log\ep$ as zero).
In terms of the coordinates
\[
\tilde y_1:=y_1+y_2\,,\qquad \tilde y_2:=y_1-y_2\,,
\]
the integral curves of the linearization
\[
\widetilde X:=-y_2\,\pd_{y_1}-y_1\,\pd_{y_2}
\]
of~$X$ are given by
\[
s(t)=s_0\,,\qquad \tilde y_1(t)= \tilde y_{10}\, e^{-t}\,,\qquad
\tilde y_2(t)= \tilde y_{20}\, e^t\,.
\]
Therefore the invariant set $\tga\equiv \{y=0\}$ of~$ X$ is normally hyperbolic because at each point
of the circle there is a one-dimensional stable component
(corresponding to the variable $\tilde y_1$ with Lyapunov exponent $-1$) where the
flow is exponentially contracting and a one-dimensional unstable component
(corresponding to $\tilde y_2$ with Lyapunov exponent $1$) where the
flow is exponentially expanding. 

It is therefore well known that the invariant circle~$\{y=0\}$ is
preserved under small perturbations of the field~$X$. More precisely,
let us take any integer~$k\geq1$, a compact set~$K\subset\cT$ enclosing the curve
$\{y=0\}$ and define the $C^k$ norm of a vector field,
\[
\|Y\|_{C^k(K)}\,,
\]
as the sum of the $C^k(K)$ norms of the components of $Y$  in the basis
$\{\pd_s,\pd_{y_1},\pd_{y_2}\}$ (this is important to avoid having to
deal with inessential factors of~$\ep$).
One then has~\cite[Theorem 4.1]{HPS} that there exists
some positive constant~$\de_1$ such that any field~$Y$ with
\begin{equation}\label{XY}
\|X-Y\|_{C^k(K)}<\de_1
\end{equation}
has a one-dimensional invariant set isotopic to the curve $\{y=0\}$,
and the distance between the corresponding isotopy~$\Theta$ and the
identity in the $C^k$ norm is of order~$\de_1$:
\[
\|\Theta-\id\|_{C^k(K)}<C\de_1\,.
\]

Since
\[
\|X-\ep\cB_J\|_{C^k(K)}<C\ep\,\log\frac1\ep\,,
\]
it follows from~\eqref{XY} that, for small enough~$\ep$, $\ep\cB_J$ (and therefore $\cB_J$) has a one-dimensional
invariant set $\tga_2$ isotopic to the curve $\tga\equiv \{y=0\}$ and contained in a
small neighborhood $\{|y|<C\ep\log\frac1\ep\}$. As the field $\cB_J$
does not vanish in the set $\{|y|<C\ep\log\frac1\ep\}$ for small enough~$\ep$,
$\tga_2$ must be a periodic integral curve of~$\cB_J$. The normal
hyperbolicity of the invariant set~$\{y=0\}$ of~$X$ implies that
$\tga_2$ is a hyperbolic periodic integral curve of~$\cB_J$, so in
particular it is robust in the sense that there is some $\de_2>0$
such that any field~$Z$ with
\begin{equation}\label{hyperb}
\|\cB_J-Z\|_{C^k(K)}<\de_2
\end{equation}
must have a periodic integral curve isotopic to~$\tga$, and the
isotopy can be chosen close to the identity in $C^k$.

\subsubsection*{Step 2: Approximation of the magnetic field created by the
  surface current distribution by the sum of the fields of a finite collection of unknotted wires}

Let us next analyze the integral curves of $J_0$, which coincide with those
of~$J$ up to a reparametrization of the curve. These are the solutions
to the system of ODEs
\begin{equation*}
\dot s= 2\cos 2\te\,,\qquad \dot\te= \frac1\ep\,,
\end{equation*}
with an arbitrary initial condition $(s_0,\te_0)$, that is,
\begin{equation}\label{stet}
 s(t)=s_0 +\ep\,\sin \Big( 2\te_0+ \frac {2t}\ep\Big)\,,\qquad \te(t)=\te_0+ \frac t\ep\,,\,.
\end{equation}
These curves are all periodic with period $2\pi\ep$. Geometrically,
for small~$\ep$ this integral curve is a small deformation of, and isotopic
to, the circle $\{s=s_0\}$ contained in the torus $\{r=1\}$. This
curve is obviously isotopic to the unknot.

Lemmas~\ref{L.conv} and~\ref{L.from} then ensure that there is a finite collection of periodic
integral curves $\{\ga_k\}_{k=1}^N$ of~$J$ such that the sum of the
magnetic fields that they create,
\[
Y:=\sum_{k=1}^N B_{\ga_k}\,,
\]
is close to $\cB_J$ in the set~$K$ modulo multiplication by a positive
constant $c:=c_0/N$:
\begin{equation}\label{BjY}
\|\cB_J-cY\|_{C^k(K)}<\frac{\de_2}3\,.
\end{equation}
This collection of curves is depicted in Figure~\ref{f1}.

\begin{figure}[t]
  \centering
 \includegraphics[scale=0.6,angle=0]{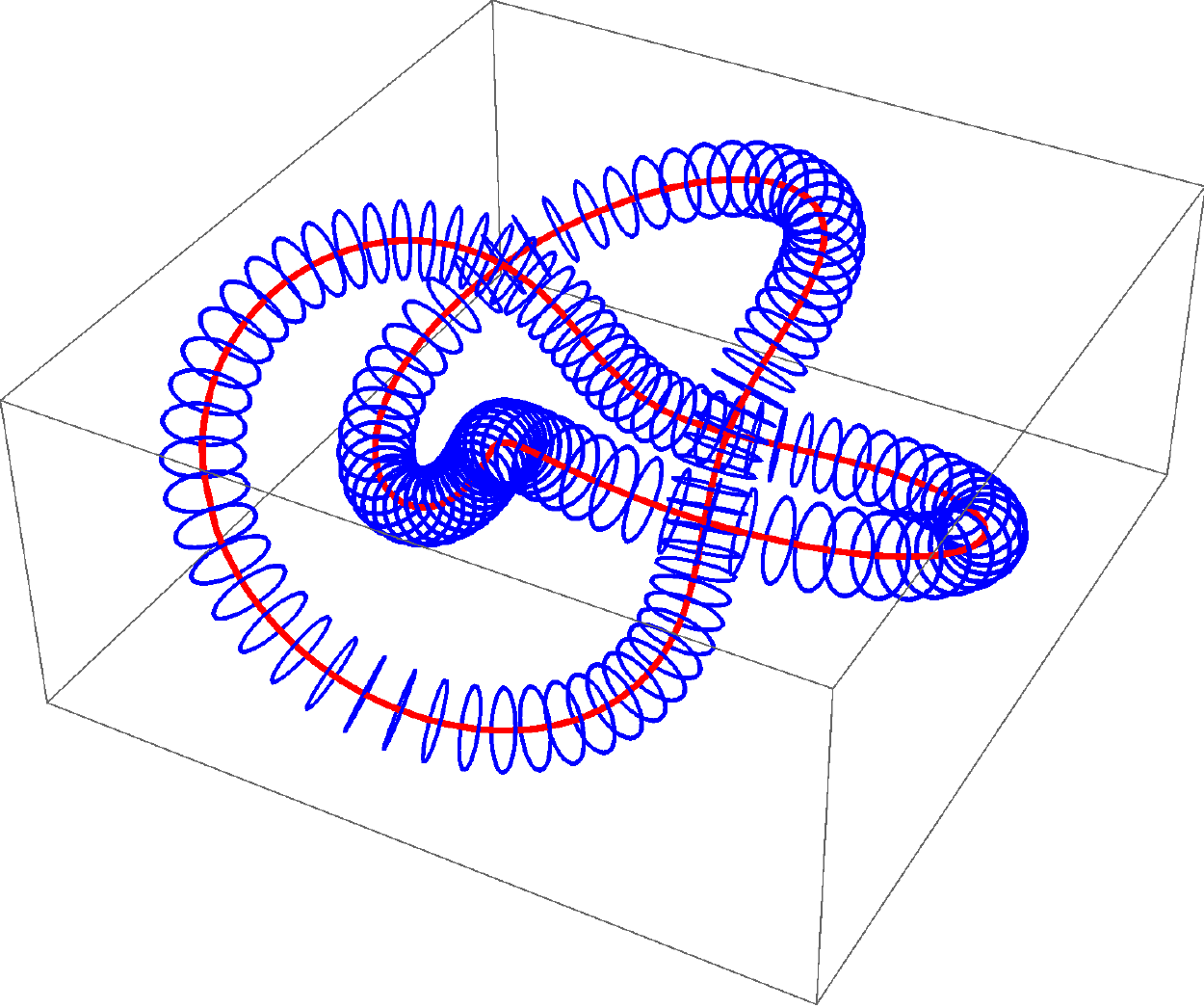}
\caption{A collection of closed wires (represented as thin blue lines) the sum of whose
  magnetic fields has a magnetic field diffeomorphic to the core knot
  (in red; in this case, a trefoil).}\label{f1}
\end{figure}

\subsubsection*{Step 3: Replacing the finite collection of unknotted
  wires by a single unknot $\Ga_\de$}

Let us denote by
\[
d\La:=\sum_{k=1}^N \dot\ga_k (t)\, dt
\]
the vector-valued measure associated with the above periodic integral
curves, so that the field~$Y$ can be written as
\[
Y(x)=\int\frac{d\La(x')\times (x-x')}{4\pi|x-x'|^3}\,.
\]
For concreteness, we will henceforth assume that the curves $\ga_k$
are parametrized as in~\eqref{stet}, so $t$ ranges over $(0,2\pi\ep)$
in each curve. Notice that the measure~$d\La$ is independent of the
way the curves are parametrized, provided that the orientation is preserved.

We will need the following observation. Let $I_1,\dots, I_N$ be intervals of length at
most~$\de$ and let us denote by $d\La_\de$ the measure obtained
from~$d\La$ after removing the intervals $I_k$ from the curves. That
is, for any continuous vector-valued function~$F$ we set
\[
\int F\cdot d\La_\de:=\sum_{k=1}^N \int_{0<t<2\pi\ep,\;  t\not\in
  I_k}F(\ga_k(t))\cdot \dot\ga_k (t)\, dt\,.
\]
Then $d\La_\de$ obviously converges weakly to $d\La$ as $\de\to0$, for
any choice of the intervals~$I_k$. 

Another useful observation is the following. Let $p,q$ be
points in~$\RR^3\backslash K$. Then, given any other pair of points $\tilde
p,\tilde q$
with
\[
|p-\tilde p|+|q-\tilde q|<\de\,,
\]
it is clear that one can choose curves
$\Ga_0,\widetilde\Ga_0:[0,1]\to\RR^3\minus K$ such that
\[
\Ga_0(0)=p\,,\qquad \Ga_0(1)=q\,,\qquad \widetilde\Ga_0(0)=\tilde p\,,\qquad
\widetilde\Ga_0(1)=\tilde q\,.
\]
and the distance between the curves is of order~$\de$ is the sense that
\begin{equation}\label{GG}
\|\Ga_0-\widetilde \Ga_0\|_{C^k([0,1])}<C\de\,.
\end{equation}
An immediate consequence of~\eqref{GG} is that, if we
reverse the orientation of $\widetilde\Ga_0$ by setting $\widehat
\Ga_0(t):=\widetilde\Ga_0(1-t)$, then the vector-valued measure
\[
\dot\Ga_0(t)\, dt + \dot{\widehat\Ga}_0(t)\, dt
\]
converges weakly to zero as $\de\to0$.

With these two observations, for each $\de>0$ we will construct a closed
oriented curve $\Ga$ (piecewise smooth, although eventually we will smooth things out)
such that the associated measure 
\[
\dot\Ga(t)\, dt
\]
converges weakly to $d\La$ as $\de\to0$. To this end, in each curve
$\ga_k$ we will fix two distinct points $p_k,q_k$. For each~$\de>0$,
let us take two points $\tilde p_k,\tilde q_k$ satisfying
\[
|p_k-\tilde p_k|+|q_k-\tilde q_k|<\de\,.
\]
The second observation above ensures that one can connect the points
$(p_k,\tilde p_k)$ with $(q_{k+1},\tilde q_{k+1})$ through curves
$\Ga_k,\widehat\Ga_k:[0,1]\to\RR^3$ that  do not intersect one another and
such that the measure
\[
\sum_{k=1}^N\big( \dot\Ga_k(t)\, dt + \dot{\widehat \Ga}_k(t)\, dt\big)
\]
converges weakly to zero as $\de\to0$. Here $k$ ranges from 1 to~$N$
and we identify $N+1$ with 1. Let us define the piecewise smooth
curve $\Ga$ as the union of the curves $\Ga_k$, $\widehat\Ga_k$
and the integral curves $\ga_k$, from which we remove the $2N$ arcs of the curves of length of
order~$\de$ connecting the points $p_k$ with~$\tilde p_k$ and $q_k$
with~$\tilde q_k$. It can be easily seen that one can choose the orientation of the
curves $\Ga_k,\widehat\Ga_k$ so that the curve~$\Ga$ has a well
defined orientation, which coincides with the orientation of
each~$\ga_k$ (cf.~Figure~\ref{f2}).

\begin{figure}[t]
  \centering
 \includegraphics[scale=0.8,angle=0]{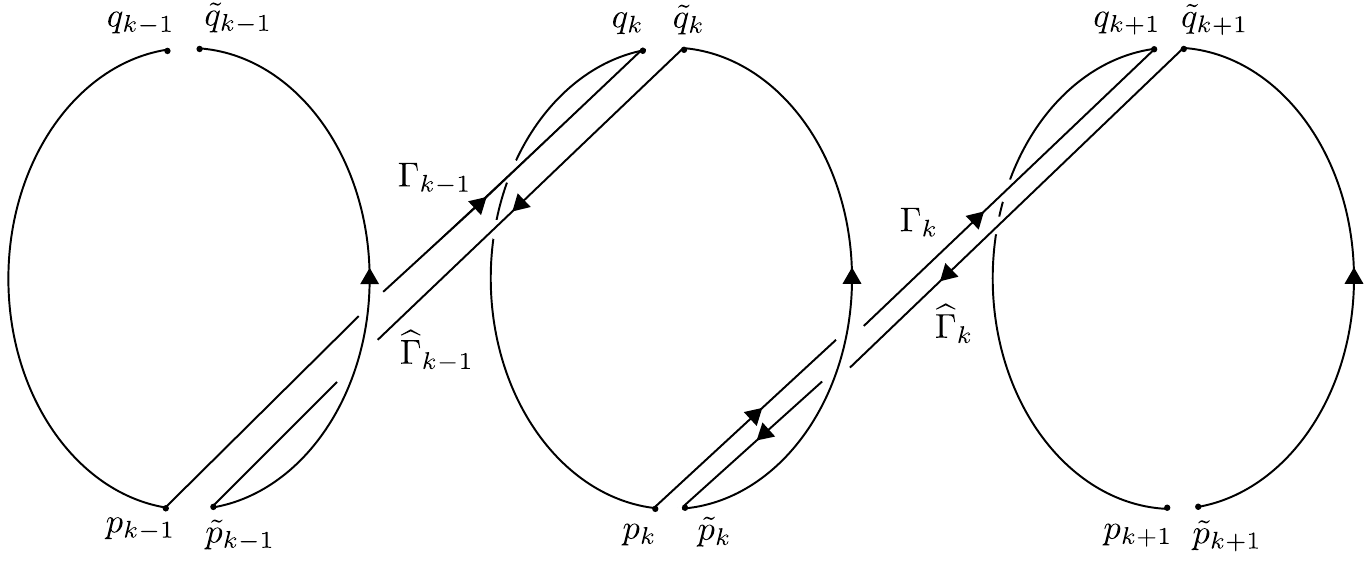}
\caption{Passing from a collection of unknots to a single unknot.}\label{f2}
\end{figure}

The two observations that we made above then imply that the measure
$\dot\Ga(t)\, dt$ converges weakly to $d\La$ as $\de\to0$, so we infer from
Lemma~\ref{L.conv} that the magnetic field created by~$\Ga$ is
close to~$Y$ in the sense that
\begin{equation}\label{BGaY}
\|B_{\Ga}-Y\|_{C^k(K)}<\frac{\de_2}{3c}
\end{equation}
whenever the constant~$\de$ is small enough. Furthermore, since
$\Ga$ has been constructed as the connected sums of the
unknots $\ga_1,\dots,\ga_N$, it is standard that it is also an unknot.

\subsubsection*{Step 4: From the unknot~$\Ga$ to $\ga$ through a connected
  sum taking place far from the magnetic line}

Let us take a large number~$R$ that will be fixed later. Translating
the curve~$\ga$ if necessary, we can assume that the distance
between~$\ga$ and the set~$K$ is at least~$R$, and that $\ga$ does not
intersect~$\Ga$.

Let us fix points $P\in\Ga$, $Q\in\ga$. For small~$\de$, let us
take another couple of points $\tilde P\in\Ga$, $\tilde Q\in\ga$
with
\[
|P-\tilde P|+|Q-\tilde Q|<\de_3
\]
for a small enough constant $\de_3$.  The second observation in
Step~3 ensures that there are oriented curves $\Ga_0,\widehat\Ga_0$
connecting the points $P,\tilde P$ with $Q,\tilde Q$, respectively,
and such that the measure
\[
\dot\Ga_0(t)\, dt + \dot{\widehat\Ga}_0(t)\, dt
\]
converges weakly to~0 as $\de_3\to0$. One can obviously assume that
the distance from these curves to the set~$K$ is uniformly bounded
away from zero. We can now define a piecewise smooth curve $\Ga'$ as
the union of the curves $\Ga_0$, $\widehat \Ga_0$, $\Ga$ and~$\ga$,
without the two arcs of length of order~$\de_3$ that connect the
points $P$ and $Q$ with $\tilde P$ and $\tilde Q$, in each case. We
choose the orientation of $\Ga'$ so that it coincides with that of
$\ga$ and~$\Ga$.

It follows from the construction that $\Ga'$ is isotopic to~$\ga$
(because it is the connected sum of~$\ga$ with an unknot) and
that 
\[
\dot \Ga'(t)\, dt\to \dot\Ga(t)\, dt+ \dot\ga(t)\, dt
\]
as $\de_3\to0$. By Lemma~\ref{L.conv}, one then has
\[
\lim_{\de_3\to0} \|B_{\Ga'}-B_\Ga-B_\ga\|_{C^k(K)}=0\,.
\]
Since the distance between $K$ and $\ga$ is at least~$R$, it is
clear that
\[
\|B_\ga\|_{C^k(K)}\leq \frac C{R^2}\,,
\]
so $B_{\Ga'}$ converges to $B_\Ga$ on~$K$ as $R\to\infty$ and
$\de_3\to0$. For convenience, let us denote by $\Ga''$ the curve that
one obtains by slightly rounding off the corners of~$\Ga'$, which can then
be chosen (for large~$R$ and small~$\de_3$) to satisfy
\[
\|B_{\Ga''}-B_\Ga\|_{C^k(K)}<\frac{\de_2}{3c}\,.
\]
By Equations~\eqref{BjY} and~\eqref{BGaY}, it then
follows that
\begin{align*}
\|\cB_J-cB_{\Ga''}\|_{C^k(K)}&\leq
\|\cB_J-cY\|_{C^k(K)}+c\|B_\Ga-Y\|_{C^k(K)}+c\|B_\Ga-B_{\Ga''}\|_{C^k(K)}\\
&<\de_2\,,
\end{align*}
so the condition~\eqref{hyperb} ensures that the magnetic
field~$B_{\Ga''}$ generated by the wire~$\Ga''$ (which is isotopic
to~$\ga$) has a periodic magnetic line that is isotopic to (and
actually a $C^k$~small deformation of) $\tga$. Theorem~\ref{T1} then follows.

\begin{remark}\label{R.many}
It is worth noticing that, although we have chosen a very concrete
current~$J$ in the proof of Theorem~\ref{T1}, the same argument works
in much greater generality. In particular, the argument goes through
for any current of the form considered in Lemma~\ref{L.asympt}
provided that the function $G(s)$ does not vanish and the Fourier
coefficients of $F(\te)$ satisfy
\[
a_1=b_1=0\,, \qquad a_2^2+b_2^2\neq0\,.
\]
\end{remark}

\section{Existence of knotted magnetic lines for generic knotted wires}
\label{S.T2}

In this section we will prove Theorem~\ref{T2}. For concreteness, let us take again the tangent vector field~$J$ on the
toroidal surface $\surf$ given by~\eqref{defJ2}, where~$\ep$ is a small
constant. 

Our objective is to show that there are periodic curves
$\Ga_n:\RR/n\ZZ\to\surf$ satisfying the hypotheses of
Lemma~\ref{L.curve} for the field~$J$ that are isotopic to~$\ga$. Notice that, as the
curve lies on~$\surf$, the $C^0$~norm of the difference between the isotopy and the
identity is of order~$\ep$, and can therefore be made as small as one
wishes. Furthermore, we will choose the curve so that the number of
intersection points $\#(\Ga_n\cap \cC)$, as defined in
Lemma~\ref{L.curve}, is precisely~$n$. Lemma~\ref{L.curve} then
implies that there is a constant $c_0$ such that the measure
\[
\frac{c_0}n\, \dot\Ga_n(t)\, dt
\]
converges weakly to $J\, dS$, so Lemma~\ref{L.asympt} ensures that
for any compact set~$K$ that does not cut~$\surf$ one has
\begin{equation}\label{eqnueva2}
\lim_{n\to\infty}\bigg\|\frac{c_0}n B_{\Ga_n}-\cB_J\bigg\|_{C^k(K)}=0\,.
\end{equation}
Since we proved in Step~1 of Section~\ref{S.T1} that the field $\cB_J$ as a
hyperbolic periodic magnetic line isotopic to~$\ga$ and $C^k$~close to
it, it stems from~\eqref{eqnueva2} that
for large enough~$n$
the magnetic field $B_{\Ga_n}$ also has a periodic magnetic line
isotopic and $C^k$~close to~$\ga$. Hence Theorem~\ref{T2} will
then follow once we construct the curve~$\Ga_n$.

The construction of the curve $\Ga_n$ is simpler in the coordinates
$(\al,\si)$ introduced in~\eqref{defPsi}. We will henceforth use the notation 
developed in the proof of Lemma~\ref{L.from} without further mention. Let us consider a sequence of points 
\begin{equation}\label{alkfin}
\{\al_k\}_{k=1}^\infty\subset\TT^1
\end{equation}
that are uniformly distributed with respect to the probability
measure~\eqref{measure}. We can safely assume that $\al_k\neq\al_{k'}$
for all $k\neq k'$. For each~$n$, let us write
\[
\{\al_k:1\leq k\leq n\}=\{\al_{n,k}:1\leq k\leq n\}\,,
\]
where $\al_{n,k}$ is a relabelling of the $n$~first points $\al_k$
chosen so that, identifying the points in $\TT^1$ with numbers
in~$[0,1)$, one has
\[
0\leq \al_{n,1}<\al_{n,2}<\cdots <\al_{n,n}\,.
\]
Since the probability measure~\eqref{measure} is absolutely
continuous, the sequence~\eqref{alkfin} is dense on~$\TT^1$, so the
difference
\[
[0,1)\ni \De_{n,k}:=\al_{n,k+1}-\al_{n,k}\mod 1\,,
\]
understood as a number in~$[0,1)$ where we are using the convention
$\al_{n,n+1}\equiv \al_{n,1}$, must tend uniformly to zero in the
sense that
\begin{equation}\label{limDe}
\lim_{n\to\infty}\max_{1\leq k\leq n}\De_{n,k}=0\,.
\end{equation}

Take a smooth function $\chi:\RR\to\RR$ such that 
\[
\chi(t)=\begin{cases}
0 &\text{for }t\leq 0\,,\\
1 &\text{for }t\geq 1\,.
\end{cases}
\]
Let us define smooth curves $\widetilde\Ga_n:\RR/
n\ZZ\to\TT^2$ in
terms of the coordinates $(\al,\si)$ as
$\widetilde\Ga_n(t):=(\al_n(t),\si_n(t))$, where
we set
\begin{align*}
\al_n(t)&:=\al_{n,1} +\sum_{k=1}^n\De_{n,k} \chi(t+1-k)&\mod 1\,,\\
\si_n(t)&:= t &\mod 1\,.
\end{align*}
It is clear that $\widetilde\Ga_n$ has period~$n$ and that in each period
the curve winds once along the coordinate~$\al$ and $n$~times along the
coordinate~$\si$. Moreover,
\[
\dot{\widetilde\Ga}_n(t)=\bigg(\sum_{k=1}^n\De_{n,k}\, \chi'(t+1-k)\bigg)\, \pd_\al+\pd_\si\,,
\]
where $\chi'$ denotes the derivative of~$\chi$. Since at most one of
the functions $\chi'(t+1-k)$ can be nonzero at any time~$t\in\RR$, it is
apparent from~\eqref{limDe} that
\begin{equation}\label{estdotGa}
\lim_{n\to\infty}\sup_{t\in\RR} \big|\dot{\widetilde\Ga}_n(t)-\pd_\si\big|=0\,.
\end{equation}

Finally, let us now define the curve $\Ga_n:\RR/n\ZZ\to\surf$ as
\[
\Ga_n(t):=\Psi^{-1}\widetilde\Ga_n(t)\,.
\]
It is not hard to see that, as $\widetilde\Ga_n$ winds once in the $\al$-direction and $n$~times in the
$\si$-direction, the curve~$\Ga_n$ is a $(1:n')$ cable over the core
curve~$\ga$, where 
\[
n':=n-N_0
\]
with $N_0$ a fixed number. In particular, for any large enough~$n$,
$\Ga_n$ is isotopic to~$\ga$. For this we will need to compute the
expression of the curves in the Frenet coordinates and to take into
account the curve's own twist. The reason is that, as changes of
coordinates in the torus do not necessarily come from an ambient
diffeomorphism, one cannot use arbitrary coordinates on the torus to
check the isotopy type of a curve.

To check this, we start by taking the set~$\cC$ as
\[
\cC:=\{r=\ep\,,\; \te=0\}\,,
\]
so the coordinate $\al:\cC\to\TT^1$ can be chosen as
\[
\al:=\frac s\ell\,,
\]
where we recall that $\ell$ denotes the length of the curve~$\ga$. The
equation for the integral curves of~$J$,
\begin{align*}
\dot s&= \frac{2\cos 2s}{1-\ep\ka(s)\, \cos\te}\,,\\
\dot \te&= \frac1{\ep(1-\ep\ka(s)\, \cos\te)}\,,
\end{align*}
implies that in terms of the time variable $t'$ defined by the ODE
\[
\frac{dt'}{dt}=\frac1{2\pi\ep[1-\ep\ka(s(t))\, \cos\te(t)]}\,,\qquad t'|_{t=0}=0\,,
\]
the integral curves are given by
\[
s=s_0+\ep \sin(2\te_0+4\pi t')\,,\qquad \te=\te_0+2\pi t'\,.
\]
Since
\[
t'=\frac {t\, (1+O(\ep))}{2\pi\ep}\,,
\]
it stems that the period is
\[
T=2\pi\ep+O(\ep^2)\,,
\]
so 
\[
\tJ=T\, J= 2\pi\, \pd_\te+O(\ep)\,.
\]
Hence the 
variable~$\si$ can be written in terms of~$(s,\te)$ as
\[
\si=\frac\te{2\pi}+O(\ep)\,.
\]
In view of the expression of $(\al,\si)$ in terms of $(s,\te)$, it is
apparent that the curve $\Ga_n$ winds once along the
coordinate $s$ and $n$~times along the coordinate~$\te$. The coordinates
$(s,\te)$ correspond to the Frenet frame, which is well known~\cite{Pohl} to twist
$N_0$ times along the curve~$\ga$, where 
\[
\pm N_0=\frac1{2\pi}\int_0^{\ell}\tau(s)\, ds+\text{Writhe}(\ga)
\]
is the total torsion of the curve plus its writhe (the sign here
depends of the orientation of the frame). Hence we infer that
$\Ga_n$ is a $(1:n')$ cable over~$\ga$, so it is isotopic
to~$\ga$ (see Figure~\ref{f3}).

\begin{figure}[t]
  \centering
 \includegraphics[scale=0.6,angle=0]{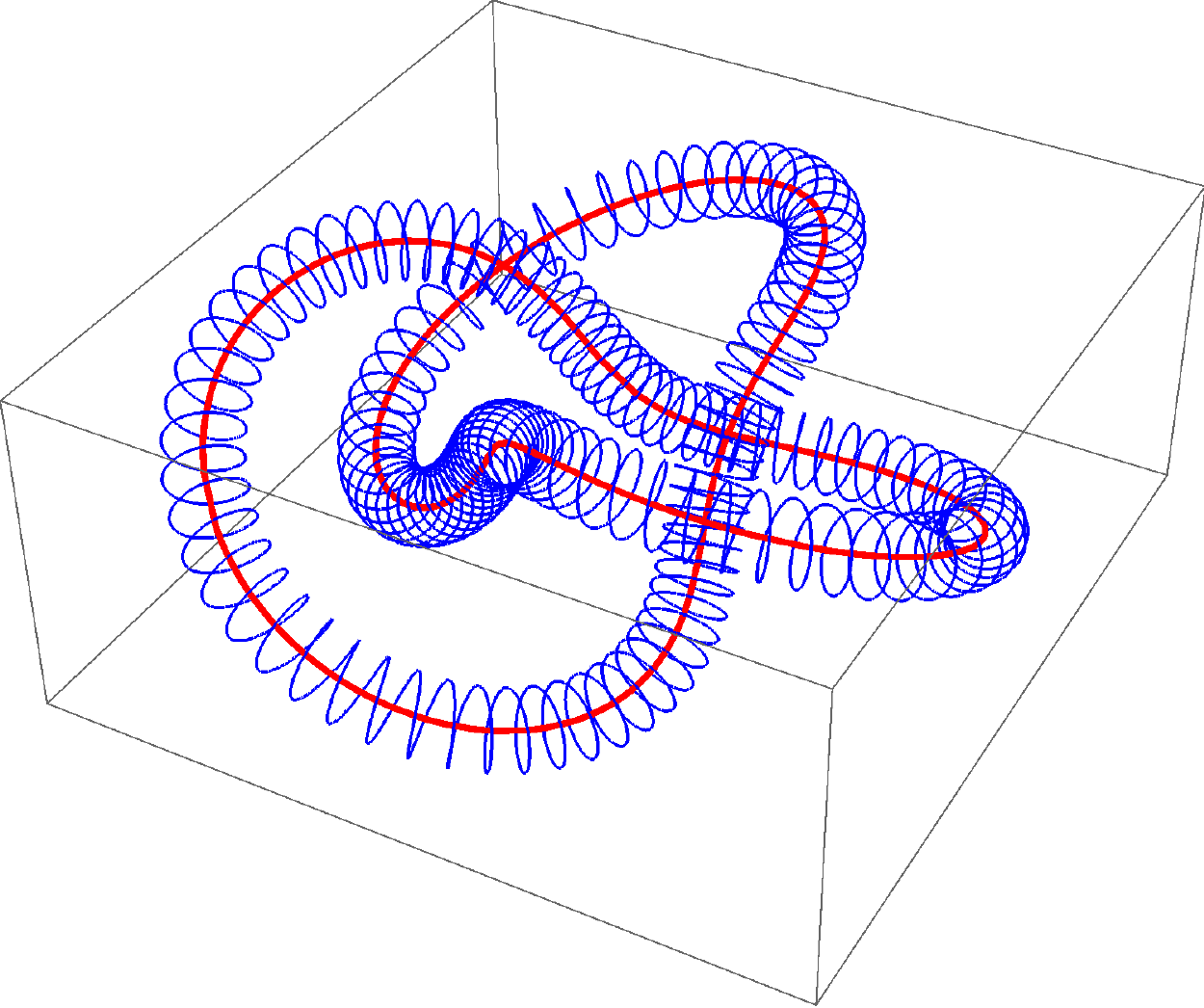}
 \caption{A wire (thin blue line) isotopic to the core knot (in this case, a
   trefoil) and $C^0$~close to it that has a magnetic line (in red)
   isotopic to the core knot. The wire and the core curve can be
   chosen arbitrarily close.}\label{f3}
\end{figure}

By construction, the intersection of $\Ga_n$ with the set~$\cC$ is
the image under the diffeomorphism $\Theta^{-1}$ of the points
$\{\al_k:1\leq k\leq n\}$, so as $n\to\infty$ they are distributed with
respect to the measure~\eqref{measure}. Moreover, since the
push-forward of the field~$\tJ$ under $\Psi$ is precisely $\pd_\si$,
it follows from~\eqref{estdotGa} that
\[
\lim_{n\to\infty}\sup_{\si\in\RR}\big|\dot\Ga_n(\si) - \tJ(\Ga_n(\si))\big|=0\,.
\]
Hence $\Ga_n$ is a sequence of curves that has the properties that we
required above, so Theorem~\ref{T2} follows.

% Using the notation that we
% developed in the proof of Lemma~\ref{L.from} without further mention,
% we shall then push forward the field~$J$ with the diffeomorphism
% $\Psi:\surf\to\TT^2$ to write
% \[
% \Psi_*\tJ=\pd_\si\,,
% \]
% where $\tJ$ was defined in terms of~$J$ in~\eqref{tJ}.

\begin{remark}
Although we have taken a concrete example of current field~$J$ for which all
the computations can be made in a very explicit way, the argument
carries over verbatim to a much more general class of fields~$J$. In
particular, sufficient conditions for the argument to remain valid are
the following:
\begin{enumerate}
\item The integral curves of~$J$ are all small deformations of (and
  isotopic to) the circles $\{s=\text{constant}\}$, in the coordinates
  that we defined on the surface. (In fact, while the fact that the integral
  curves are all periodic is key, the condition that they are small
  deformations of the circles $\{s=\text{constant}\}$ can be relaxed
  significantly, as it is only used to control the isotopy type of the
  curve.)

\item The magnetic field $\cB_J$ has a hyperbolic periodic magnetic line
  isotopic to and $C^k$~close to~$\ga$. As we saw in
  Section~\ref{S.T1}, a sufficient condition for this is that the
  functions~$F$ and~$G$ that appear in the field~$J$ satisfy the
  conditions of Remark~\ref{R.many}.
\end{enumerate}
\end{remark}

\section*{Acknowledgments}

The authors are supported by the ERC Starting Grants~633152 (A.E.) and~335079
(D.P.-S.). This work is supported in part by the
ICMAT--Severo Ochoa grant
SEV-2015-0554.

\bibliographystyle{amsplain}

\end{document}